\documentclass{amsart}

\usepackage[utf8]{inputenc}
\usepackage[russian,english]{babel}

\usepackage[backend=bibtex,style=alphabetic,natbib=false,url=false,sorting=nyt,doi=false,backref=false]{biblatex}
\usepackage{fullpage}
\usepackage{graphicx}

\usepackage{tikz} 
 \usetikzlibrary{arrows}
\usetikzlibrary{intersections}
\usetikzlibrary{calc}
\usepackage[pdftex,bookmarks,colorlinks,breaklinks]{hyperref}  
\usepackage{color}
\definecolor{dullmagenta}{rgb}{0.4,0,0.4}   
\definecolor{darkblue}{rgb}{0,0,0.4}
\hypersetup{linkcolor=red,citecolor=blue,filecolor=dullmagenta,urlcolor=darkblue} 

\newcommand{\x}{{\mathbf x}}
\newcommand{\y}{{\mathbf y}}
\renewcommand{\c}{{\mathbf c}}

\renewcommand{\k}{{\mathbf k}}
\newtheorem{thm}{Theorem}[section]

\newtheorem{prop}[thm]{Proposition}
\newtheorem{lem}[thm]{Lemma}
\theoremstyle{remark}
\newtheorem{rem}[thm]{Remark}

 \graphicspath{{../figs/}} 

\begin{document}
\title{An Isoperimetric inequality for \\  an integral operator on flat tori}

\author{Braxton Osting}
\address{Department of Mathematics, 
University of Utah, 
Salt Lake City, UT 84112, USA}
\email{osting@math.utah.edu}
\thanks{}

\author{Jeremy Marzuola}
\address{Department of Mathematics, University of North Carolina, Chapel Hill, NC  27599, USA}
\email{marzuola@math.unc.edu}
\thanks{}

\author{Elena Cherkaev}
\address{Department of Mathematics, 
University of Utah, 
Salt Lake City, UT 84112, USA}
\email{elena@math.utah.edu}
\thanks{}



\date{\today}

\begin{abstract}
We consider a class of Hilbert-Schmidt integral operators with an isotropic, stationary kernel acting on square integrable functions defined on flat tori. For any fixed kernel which is positive and decreasing, we show that among all unit-volume flat tori, the equilateral torus maximizes the operator norm and the Hilbert-Schmidt norm. 
\end{abstract}

\maketitle

\section{Statement of  results and discussion}
Let $T_{a,b}$ be a unit-volume flat torus with parameters $(a,b)\in U \subset \mathbb{R}^2$ with 
\begin{equation}
\label{eqn:U}
U := \left\{ (a,b) \in \mathbb R^2 \colon b>0, \ a \in [ 0,1/2  ], \ \text{and} \ a^2 + b^2 \geq 1 \right\}.
\end{equation}
The parametrization we use is fairly standard and will be introduced in detail in Section \ref{sec:proof}. The set $U$ is illustrated in Figure \ref{fig:FundDom} and, in this notation, the square torus has parameters $(a,b) = (0,1)$ and the equilateral torus has parameters $(a,b) = \left(\frac{1}{2}, \frac{\sqrt3}{2} \right)$.
 Consider the integral operator $A_f \colon L^2(T_{a,b}) \to L^2(T_{a,b})$ with  isotropic stationary kernel, $f\colon \mathbb R \to \mathbb R$, given by 
$$
A_f \psi (\x) := \int_{T_{a,b}} f\left(d^2(\x,\y)\right) \psi(\y) \ d\y. 
$$
Here, $d$ is the geodesic distance on $T_{a,b}$ and $f$ is evaluated at the squared distance. 
If 
\begin{equation} \label{eq:IntKer}
\int_{T_{a,b}} \int_{T_{a,b}} |f\left( d^2(\x,\y)\right)|^2 \ d\x d\y < \infty,
\end{equation} 
then $A_f$ is a Hilbert-Schmidt integral operator, hence bounded and compact. 
Let $\| \cdot \|_{L^2(T_{a,b})}$ and $\| \cdot \|_{\mathrm{H.S.}}$ denote the  $L^2(T_{a,b}) \to L^2(T_{a,b})$ and Hilbert-Schmidt  operator norms respectively. 

\begin{thm} \label{thm:main}
Let $f$ be a fixed positive and non-increasing function satisfying \eqref{eq:IntKer}. 
Among unit-volume flat tori, $T_{a,b}$, the equilateral torus is a maximizer of both  $\| A_f \|_{L^2(T_{a,b})}$ and $\| A_f \|_{\mathrm{H.S.}}$. The square torus is a critical  point of these spectral quantities.  If $f$ is additionally assumed to be decreasing, then the equilateral torus is the unique minimizer and the square torus is a saddle point of these spectral quantities. 
\end{thm}

\begin{rem}
A self-contained proof of Theorem \ref{thm:main} is given in Section \ref{sec:proof}.   The result that the equilateral torus is a maximizer follows from the Moment Lemma of L\'aszl\'o Fejes T\'oth (see Remark \ref{MomentLemma}). As we were unable to find an English version of this lemma in the literature,  we provide a translated proof of in Appendix \ref{app:momlem}.  Here, we provide a self-contained, constructive proof of Theorem \ref{thm:main} that gives more precise insight into spectral optimization problems over the family of flat tori considered.  In particular, we emphasize that we observe the square torus as a saddle point of the reported spectral quantities and gain fairly precise control over the dependence on the parameters $(a,b)$.   
\end{rem}

To prove Theorem \ref{thm:main}, we obtain explicit expressions for $\| A_f \|_{L^2(T_{a,b})}$ and $\| A_f \|_{\mathrm{H.S.}}$ and then use a  rearrangement method to show that the equilateral torus is maximal. The proof involves the rearrangement of a six-sided polygon with some special properties. Generally speaking, rearrangement of polygons is difficult, as symmetrization typically destroys  polygonal structure, {\it e.g.}, Steiner symmetrization  introduces additional vertices. In fact, as far as we know, it remains an open problem to prove the conjecture of George P\`olya and G{\'a}bor Szeg{\H{o}} \cite[Section 7]{PolyaSzego1951}, that among all $N$-gons  of given area for fixed $N\geq 5$,  the regular one has the smallest  first Laplace-Dirichlet eigenvalue; see also   \cite[Open Problem 2]{Henrot:2006}. 
Theorem \ref{thm:main}  is most similar to a result of Marcel Berger, which shows that the maximum first eigenvalue of the Laplace-Beltrami operator  over all flat tori is attained only by the equilateral torus \cite{Berger1973,LapBel2013}.  See also the work of Baernstein \cite{baernstein1997minimum}, in which heat kernels are studied over flat tori. 
Finally, the present work can in some sense be seen as complementary to the lattice optimization problems of minimizing the Epstein zeta energy \cite{rankin1953,cassels1959,diananda1964,ennola1964} 
and the theta energy  \cite{Montgomery1988}, both of which are minimized by the triangular lattice.


\section{Proof of Theorem \ref{thm:main}} \label{sec:proof}
Let $B  \in \mathbb R^{2\times 2}$ have linearly independent columns. 
The lattice generated by the basis $B$ is the set  of integer linear combinations of the columns of $B$, 
$L(B) = B(\mathbb Z^2)$.
Let $B$ and $C$ be two lattice bases. We recall that $ L(B)$ and $ L(C)$ are \emph{isometric} if there is a unimodular\footnote{A  matrix $A \in \mathbb Z^{n\times n}$ is \emph{unimodular} if $\mathrm{det} A  = \pm 1$.}  
matrix $U$ such that $B = CU$. 
The following proposition gives a parameterization of the space of two-dimensional, unit-volume lattices modulo isometry. 

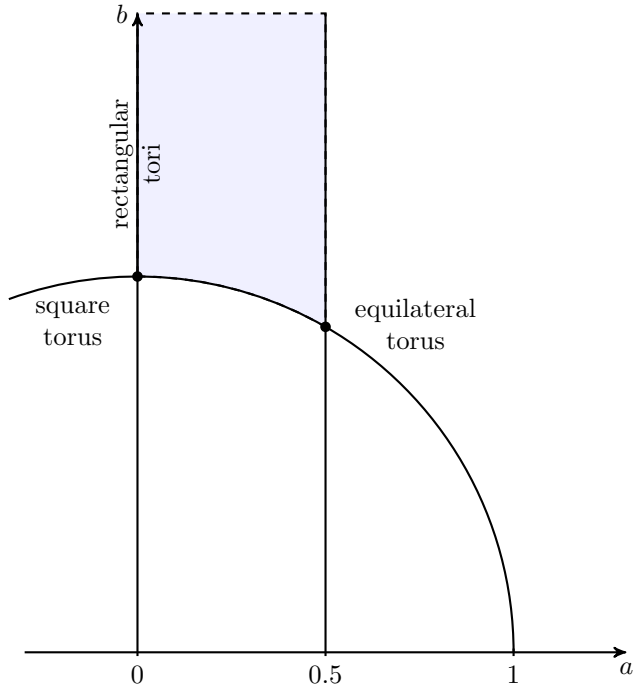
\begin{figure}[t]
\begin{center}
\begin{tikzpicture}[scale=.5,thick,>=stealth',dot/.style = {fill = black,circle,inner sep = 0pt,minimum size = 4pt}]

 \filldraw[fill=blue!6]  (0,10) arc(90:60:10cm)  -- (5,17) [dashed] -- (0,17) -- cycle;

\draw[->] (-3,0) -- (13,0) coordinate[label = {below:$a$}];
\draw[->] (0,-.1) coordinate[label={below:$0$}] -- (0,17) coordinate[label = {left:$b$}];

\draw (10,.1) -- (10,-.1) coordinate[label={below:$1$}]; 
\draw (5,-.1)coordinate[label = {below:$0.5$}] -- (5,17) coordinate(top); 
\draw (10,0) arc (0:110:10); 
 
 \draw (5,8.660254) node[dot,label={right: \begin{tabular}{c} equilateral \\torus \end{tabular}}](triLat){};
 \draw (0,10) -- (0,13) coordinate[label={center:\rotatebox{90}{\begin{tabular}{c}rectangular\\ tori \end{tabular}}}]{} --(0,17) ;
 \draw (0,10) node[dot,label={below left:\begin{tabular}{c} square \\ torus \end{tabular}}]{};

\end{tikzpicture}
\caption{The set $U$, defined in \eqref{eqn:U}, is shaded.}
\label{fig:FundDom}
\end{center}
\end{figure}

\begin{prop} \label{prop:LatticeParam}
Every two-dimensional lattice with volume one is isometric to a lattice, $L_{a,b}$, parameterized by the basis 
$$ B_{a,b} = \begin{pmatrix} \frac{1}{\sqrt b} & \frac{a}{\sqrt b} \\ 0 & \sqrt b \end{pmatrix}, $$ 
where the parameters $a, b \in U$ and $U$ is defined in \eqref{eqn:U} and illustrated in Figure \ref{fig:FundDom}.
\end{prop}  

For $(a,b) \in U$, we refer to $L_{a,b} = B_{a,b}(\mathbb Z^2)$ as the lattice generated by vector $(a,b)$.  A proof of this Proposition \ref{prop:LatticeParam} is standard, but for completeness is provided in Appendix \ref{sec:LatPara}. 

A \emph{flat torus} is a torus with a metric inherited from its representation as the quotient $\mathbb R^2 / L$, where $L$ is a lattice. Tori are isometric iff the matrices of the generating lattices are equivalent via left multiplication by an orthogonal matrix, see \cite{wolpert}.  Hence, by Proposition \ref{prop:LatticeParam} a parameterization of the unit volume flat tori is given by 
$$
T_{a,b} = \mathbb R^2 / L_{a,b}, \qquad (a,b) \in U. 
$$

Let $\Lambda_{a,b} = B_{a,b}^{-t}(\mathbb Z^2)$ be the dual lattice of $L_{a,b}$, namely the set of all vectors whose inner products with each vector in $L_{a,b}$ is an integer,
\[
\Lambda_{a,b}  = \{ \y \in \mathbb{R}^2 \colon    \y \cdot \x \in \mathbb{Z} \ \text{for all} \ \x \in L_{a,b}  \};
\]  see for instance \cite[Chapter $XIII.16$]{RSv4}. 
For $\k \in \Lambda_{a,b}$, a computation using lattice Fourier analysis shows that
$$
A_f \ e^{\imath \k\cdot\x} = \int_{T_{a,b}} f\left(d^2(\y,\c)\right) e^{-\imath \k \cdot \y } d\y  \ e^{\imath \k \cdot \x} \equiv \gamma_f(\k) \ e^{\imath \k \cdot \x}, 
$$
where $\c=\c_{a,b}$ is the center of $T_{a,b}$.  For the representative parallelogram of $T_{a,b}$ with $(0,0)$ in the bottom left corner, we have $\c_{a,b} = \left( (a+1)/(2 \sqrt{b}), 1/(2 \sqrt{b}) \right)$. 
Thus $A_f$ is diagonalized by the Fourier Transform with eigenvalues given by $ \gamma_f(\k)$, $\k \in \Lambda_{a,b}$. We observe that 
$$
\gamma_f(\k) 
= \int_{T_{a,b}} f\left(d^2(\x,\c)\right) e^{-\imath \k \cdot \x } d\x \leq 
\int_{T_{a,b}} f\left(d^2(\x,\c)\right) d\x = \gamma_f(\mathbf{0}),
$$
and thus the largest eigenvalue of $A_f$ is given by $ \gamma_f(\mathbf{0})$, 
\begin{equation} \label{eq:L2}
\| A_f \|_{L^2(T_{a,b})} = \int_{T_{a,b}} f\left(d^2(\x,\c)\right) d\x. 
\end{equation}
Equation \eqref{eq:L2} also follows from Young's inequality. 
We also compute 
\begin{align} \label{eq:HS}
\| A_f \|_{\mathrm{H.S.}} 
=  \int_{T_{a,b}}  f^2\left(d^2( \x,\c ) \right) \ d\x,
\end{align}
where we used the assumption that $|T_{a,b}| = 1$. 

From \eqref{eq:L2} and \eqref{eq:HS}, the proof of Theorem \ref{thm:main} requires us to solve the following optimization problem 
\begin{equation} \label{eq:OptProb}
\max_{(a,b) \in U}  \ J(a,b) \ \ \text{where } \ J(a,b) :=\int_{T_{a,b}} f\left(d^2(\x, \c)\right)  \ d\x. 
\end{equation}

We tile the plane with the torus $T_{a,b}$ and consider the Dirichlet-Voronoi  cell for the origin, $D_{a,b}$, defined as the set of points closer to the origin than any of the other lattice points.   We define 
\begin{align}
 \label{y1}
y_1(x) &:= \frac{1}{2\sqrt b} \left( \frac{(1-a)^2}{b} + b \right) + \frac{1-a}{b}x,  \\
 \label{y2}
y_2(x) &:= \frac{1}{2\sqrt b} \left( \frac{a^2}{b} + b \right) - \frac{a}{b}x, 
\end{align}
and
\begin{equation}
x_1 := \frac{1}{2\sqrt{b}},  \quad  x_2:= \frac{a-1/2}{\sqrt{b}}.
\end{equation}
The Dirichlet-Voronoi  cell for the origin can be explicitly written 
\begin{align*}
D_{a,b} = \{(x,y)\in \mathbb R^2 \colon  
& x \in [-x_1,x_1],  \ y\leq y_1(x) \ \forall x \in [-x_1,x_2], \ y\leq y_2(x) \ \forall x \in [x_2, x_1], \\ 
& \ y\geq -y_2(-x) \ \forall x \in [-x_1,-x_2], \ y\geq -y_1(-x) \ \forall x \in [-x_2,x_1]\}. 
\end{align*}
This Dirichlet-Voronoi cell is illustrated in Figure \ref{fig:DCell} for a square torus (left), a torus $T_{a,b}$ with  $(a,b) = \left(0.2,1.2\right)$ (center), and an equilateral torus (right).  The vertices of the Dirichlet-Voronoi cell are concyclic; the Dirichlet-Voronoi cell is a cyclic polygon with four vertices if $a=0$ and six vertices if $a \in (0,1/2]$. The inradius is given by 
$$r_1 := \frac{1}{2\sqrt{b}}$$ 
and the circumradius is given by 
$$
r_2:= \sqrt{\frac{(a^2+b^2)\left( (a-1)^2 + b^2\right)}{4 b^3}} > r_1.
$$

\begin{figure}[t!]
\begin{center}
\hspace{-1.7cm}
\includegraphics[width=.39\textwidth]{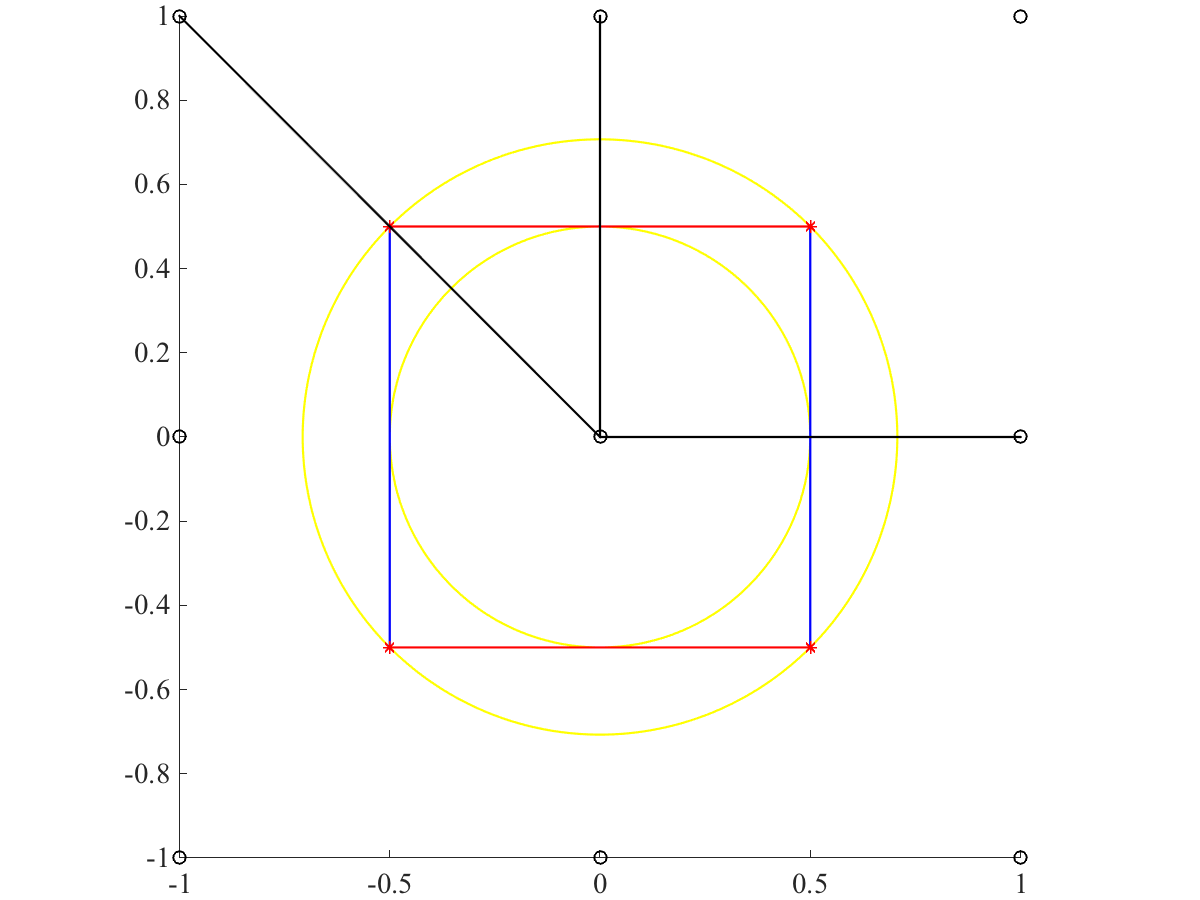}  \hspace{-1cm}
\includegraphics[width=.39\textwidth]{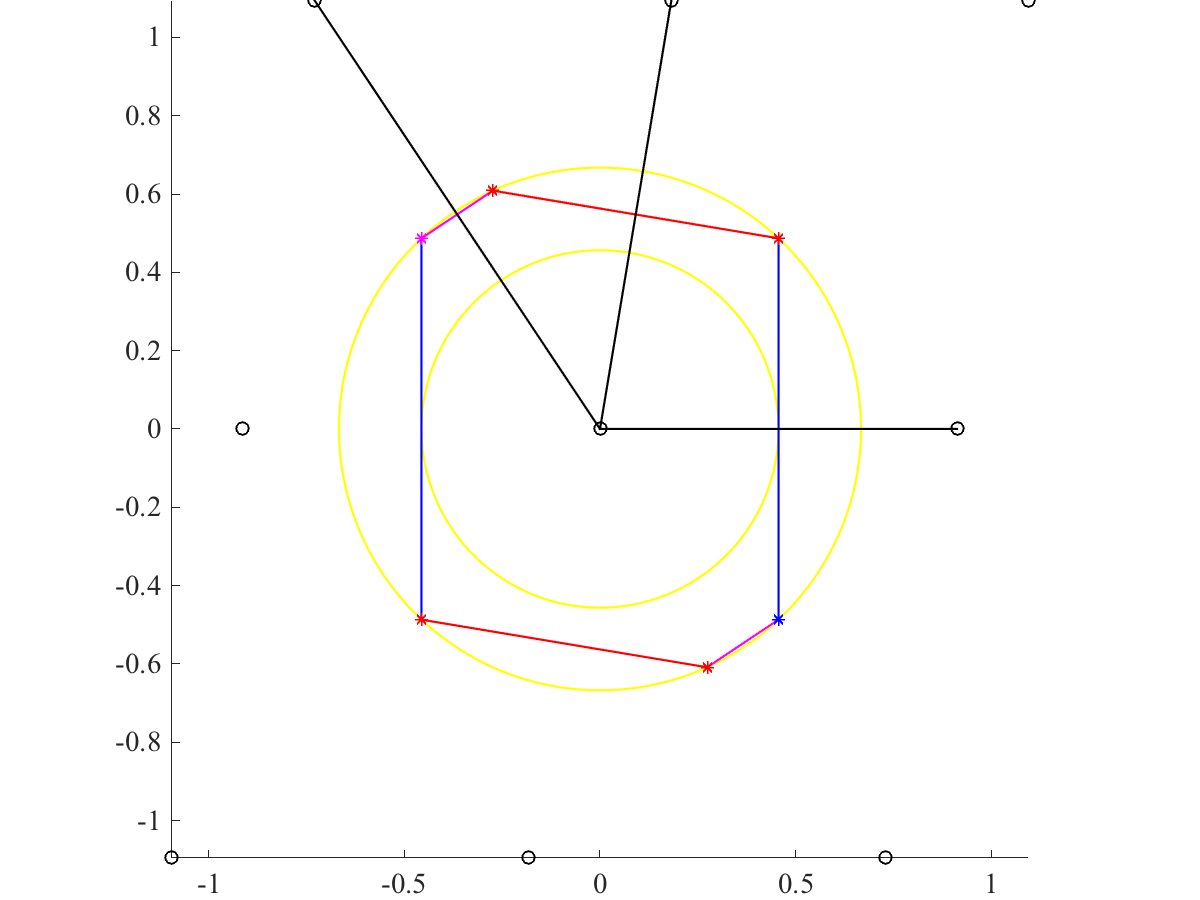} \hspace{-1cm}
\includegraphics[width=.39\textwidth]{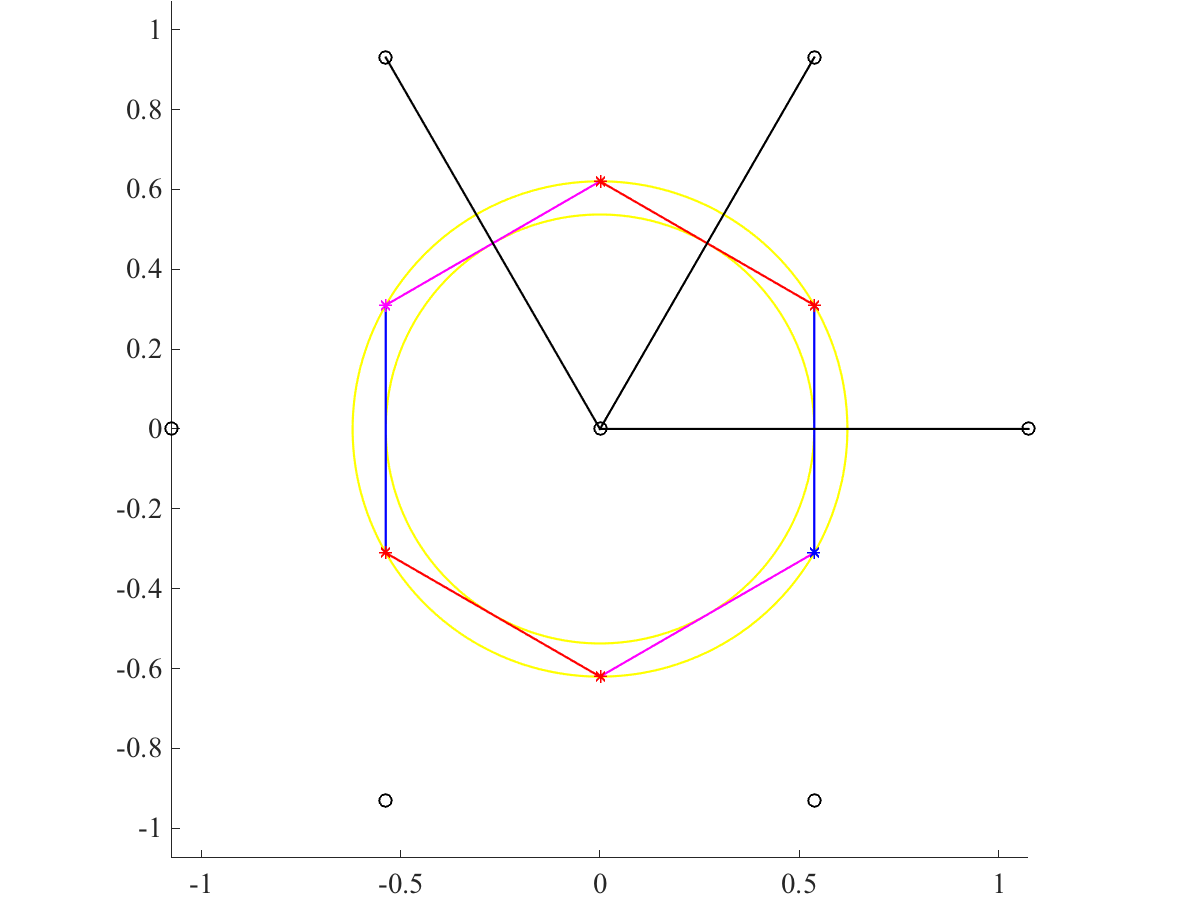}\hspace{-1.7cm}
\caption{An illustration of $D_{a,b}$, the Dirichlet-Voronoi cell of the origin, 
 for a square torus {\bf (left)}, a torus with  $(a,b) = \left(0.2,1.2\right)$ {\bf (center)}, and an equilateral torus {\bf (right)}. 
 The vertices of the Dirichlet-Voronoi cell are concyclic and both the incircle and circumcircle are plotted. }
\label{fig:DCell}
\end{center}
\end{figure}

For $(a,b) \in U$, the objective function, $J(a,b)$,  in \eqref{eq:OptProb} can be rewritten as an integral over the Dirchlet cell  where the distance is now simply the Euclidean distance to the origin,
\begin{equation} \label{eq:ObjDir}
J(a,b)
= \int_{D_{a,b}} f\left( \| \x \|^2 \right) \ d\x , \qquad \qquad (a,b) \in U.
\end{equation}

\begin{rem} \label{MomentLemma}
From \eqref{eq:ObjDir}, Theorem \ref{thm:main} now follows from the Moment Lemma of L\'aszl\'o Fejes T\'oth, see Theorem \ref{thm:MomLemma} or, for example,  \cite[p. 198]{toth1953}, \cite{toth1973}, or \cite{gruber1999}, \cite{BoroczkyCsikos2010}.  Namely, if $f$ is a non-increasing function and $D(\mathbf{p})$ any Dirichlet-Voronoi cell  of a point $\mathbf{p}\in \mathbb R^2$ with 6 vertices, we have 
$$
\int_{D(\mathbf{p})} f\left( \| \x - \mathbf{p} \|^2 \right) \ d\x \leq \int_{R} f\left( \| \x \|^2 \right) \ d\x,
$$
where $R$ is the regular $6$-gon centered at the origin with $|R| = |D(\mathbf{p})|$. This  result is intuitive; since $f$ is a positive and decreasing function, the optimal torus will be described by the parameters $(a,b) \in U$ such that the Dirichlet-Voronoi cell is most concentrated about the origin. The equilateral lattice has the Dirichlet-Voronoi cell with the most symmetry, so intuitively it is the maximizer. 
\end{rem}

\begin{rem} \label{rem:dec} We will prove the statement in Theorem \ref{thm:main} which assumes that $f$ is decreasing; the statement for $f$ non-increasing can be proved similarly. 
\end{rem}

We  complete  the proof by showing that when we continually rearrange the Dirichlet-Voronoi cell, first increasing $a$ and then decreasing $b$ while keeping $(a,b) \in U$, the objective function is strictly increasing. This implies that the maximum is attained uniquely by the equilateral torus. We proceed with a few preliminary  results, and then put them together to finish the proof. 

We use the polar symmetry ($\x\in T_{a,b} \implies -\x \in T_{a,b}$) to reduce the integral in \eqref{eq:ObjDir} to the upper half plane. We obtain 
\begin{equation} \label{eq:ObjDir2}
J(a,b) = 2 \int_{-x_1}^{x_2} \int_{0}^{y_1(x)} f\left(x^2 + y^2 \right) \ dy dx +
 2 \int_{x_2}^{x_1} \int_{0}^{y_2(x)} f\left(x^2 + y^2 \right) \ dy dx, \quad (a,b) \in U. 
\end{equation}
We  compute the partial derivatives of $J(a,b)$. The partial derivative of $J(a,b)$ with respect to the parameter $a$, is given by 
\begin{align}
 \partial_a \frac{1}{2}  J(a,b)  \notag
&= \int_{-x_1}^{x_2} f\left(x^2 + y_1^2(x)\right) \left( \partial_a y_1(x) \right) dx 
+ \int_{x_2}^{x_1} f\left(x^2 + y_2^2(x)\right) \left( \partial_a y_2(x) \right) dx \notag \\
&=\frac{1}{b}  \int_{-x_1}^{x_2} f\left(x^2 + y_1^2(x)\right) \left( \frac{a-1}{\sqrt b} - x \right) dx 
+ \frac{1}{b} \int_{x_2}^{x_1} f\left(x^2 + y_2^2(x)\right) \left( \frac{a}{\sqrt b} - x \right) dx \notag \\
&=-\frac{1}{b}  \int_{-x_1}^{x_2} f\left(x^2 + y_1^2(x)\right) \left( x- \frac{a-1}{\sqrt b}  \right) dx  \label{paJ} \\
& \quad + \frac{1}{b} \int_{x_2}^{\frac {a} {\sqrt b}} f\left(x^2 + y_2^2(x)\right) \left( \frac{a}{\sqrt b} - x \right) dx 
- \frac{1}{b} \int_{\frac {a} {\sqrt b}}^{x_1} f\left(x^2 + y_2^2(x)\right) \left( x -\frac{a}{\sqrt b}  \right) dx  . \notag 
\end{align}
The partial derivative of $J(a,b)$ with respect to the parameter $b$ is
\begin{align} 
 \partial_b \frac{1}{2}  J(a,b)  
=&  -\frac{1}{2 b^{\frac 3 2}} \int_{0}^{y_2(x_1)} f\left(\frac{1}{4b}  + y^2\right)  dy \nonumber \\
& +\int_{-x_1}^{x_2} f\left(x^2 + y_1^2(x)\right) \left( \partial_b y_1(x) \right) dx 
 + \int_{x_2}^{x_1} f\left(x^2 + y_2^2(x)\right) \left( \partial_b y_2(x) \right) dx. \label{eq:Jb}
\end{align}

\begin{lem}
For any integrable function $f$, $\partial_a J(a,b) = 0$ if either $a=0$ or $a= \frac{1}{2}$. Furthermore, the square $(a,b) = (0,1)$ and equilateral torus $(a,b) = \left(\frac 1 2, \frac{\sqrt{3}}{2} \right)$  are critical points of $J(a,b)$. 
\end{lem}
\begin{proof}
Setting $a=0$ we have that $x_2 = - x_1$ and $y_2(x) = \frac{\sqrt b}{2}$. From  \eqref{paJ}, we obtain 
$$
 \partial_a \frac{1}{2}  J(a,b) = - \frac{1}{b} \int_{-x_1}^{x_1} f\left(x^2 + \frac{b}{4}\right) \ x \ dx = 0.
$$
If, additionally $b=1$, we have from \eqref{eq:Jb} that 
$$
 \partial_b \frac{1}{2}  J(a,b) =- \frac{1}{2} \int_{0}^{x_1} f\left(\frac{1}{4} + y^2  \right) dy + \int_{-x_1}^{x_1} f\left(x^2 + \frac{1}{4} \right) \frac{1}{4} dx = 0.
$$

For $a = \frac{1}{2}$, we have $x_2 = 0$ and $y_2(x) = \frac{1}{2\sqrt{b} } \left( \frac{1}{4b} + b \right) - \frac{1}{2b} x = y_1(-x)$.  
From  \eqref{paJ}, we obtain 
\begin{align*}
 \partial_a \frac{1}{2}  J(a,b)  &=- \frac{1}{b}  \int_{-x_1}^{0} f\left(x^2 + y_1^2(x)\right) \left( \frac{1}{2\sqrt b} + x \right) dx 
+ \frac{1}{b} \int_{0}^{x_1} f\left(x^2 + y_1^2(-x)\right) \left( \frac{1}{2\sqrt b} - x \right) dx \\
&=0.
 \end{align*}
If, additionally $b=\frac{\sqrt 3}{2}$, then we have  
$y_2(x) = \frac{x_1}{b} - \frac{1}{2b} x = y_1(-x)$ and
$$\partial_b y_2(x) =  \frac{1}{2b^2}x = (\partial_b y_1) (-x). $$
Hence, 
\begin{align*}
 \partial_b \frac{1}{2}  J(a,b) 
= & -\frac{1}{2 b^{\frac 3 2}} \int_{0}^{\frac{x_1}{2b}} f\left(\frac{1}{4b} + y^2\right)  dy + \frac{1}{b^2} \int_{0}^{x_1} f\left(x^2 + y_2^2(x)\right) x dx. 
 \end{align*}
Looking at the second integral, we complete the square in the argument of $f$, 
$$
x^2 + y_2^2(x) = \frac{1}{b^2} \left( x - \frac{x_1}{2} \right)^2 + \frac{1}{4b}. 
$$
Now letting $z =  \frac{1}{b} \left( x - \frac{x_1}{2} \right)$, the second integral can then be rewritten 
\begin{align*}
 \frac{1}{b^2} \int_{0}^{x_1} f\left( x^2 + y_2^2(x) \right) x dx &=   \frac{1}{b^2} \int_{-\frac{x_1}{2b}}^{\frac{x_1}{2b}} f\left(  \frac{1}{4b} + z^2 \right)  \left( bz + \frac{x_1}{2}\right)  b dz \\
 &=  \frac{x_1}{b}  \int_{0}^{\frac{x_1}{2b}} f\left(  \frac{1}{4b} + z^2 \right)   dz, 
\end{align*}
which perfectly cancels the first integral. Thus, $ \partial_b  J(a,b) =0$. 
\end{proof}

\begin{lem} \label{lem:Ja>0}
For any positive and decreasing function $f$, $\partial_a J(a,b) >  0$ for $(a,b) \in U$ with $a\in(0,1/2)$. 
\end{lem}

\begin{proof} Fix $(a,b) \in U$ with $a\in(0,1)$. 
From \eqref{paJ}, we write 
$$
 \partial_a \frac{1}{2}  J(a,b) \equiv - I_1 + I_2 - I_3, 
$$
where
\begin{align*}
I_1 &:= \frac{1}{b}  \int_{-x_1}^{x_2} f\left(x^2 + y_1^2(x)\right) \left( x- \frac{a-1}{\sqrt b}  \right) dx \\
I_2 & := \frac{1}{b} \int_{x_2}^{\frac {a} {\sqrt b}} f\left(x^2 + y_2^2(x)\right) \left( \frac{a}{\sqrt b} - x \right) dx \\
I_3 & := \frac{1}{b} \int_{\frac {a} {\sqrt b}}^{x_1} f\left(x^2 + y_2^2(x)\right) \left( x -\frac{a}{\sqrt b}  \right) dx. 
\end{align*}
The integrals defined in $I_1$, $I_2$ and $I_3$ are non-negative. These integrals have a nice geometric picture, which we illustrate in Figure \ref{fig:ChangeDcell}. Here, we plot the Dirichlet-Voronoi cell for $b = 1$ and two different values of $a$: $a=0.2$ and $a=0.3$. The three integrals each correspond to a piece of the Dirichlet-Voronoi cell that is being added or removed as the parameter $a$ is varied. We now transform each interval of integration to $[-1,1]$ so we can compare the magnitude of the three integrals. 

\begin{figure}[t!]
\begin{center}
\includegraphics[width=.6\textwidth]{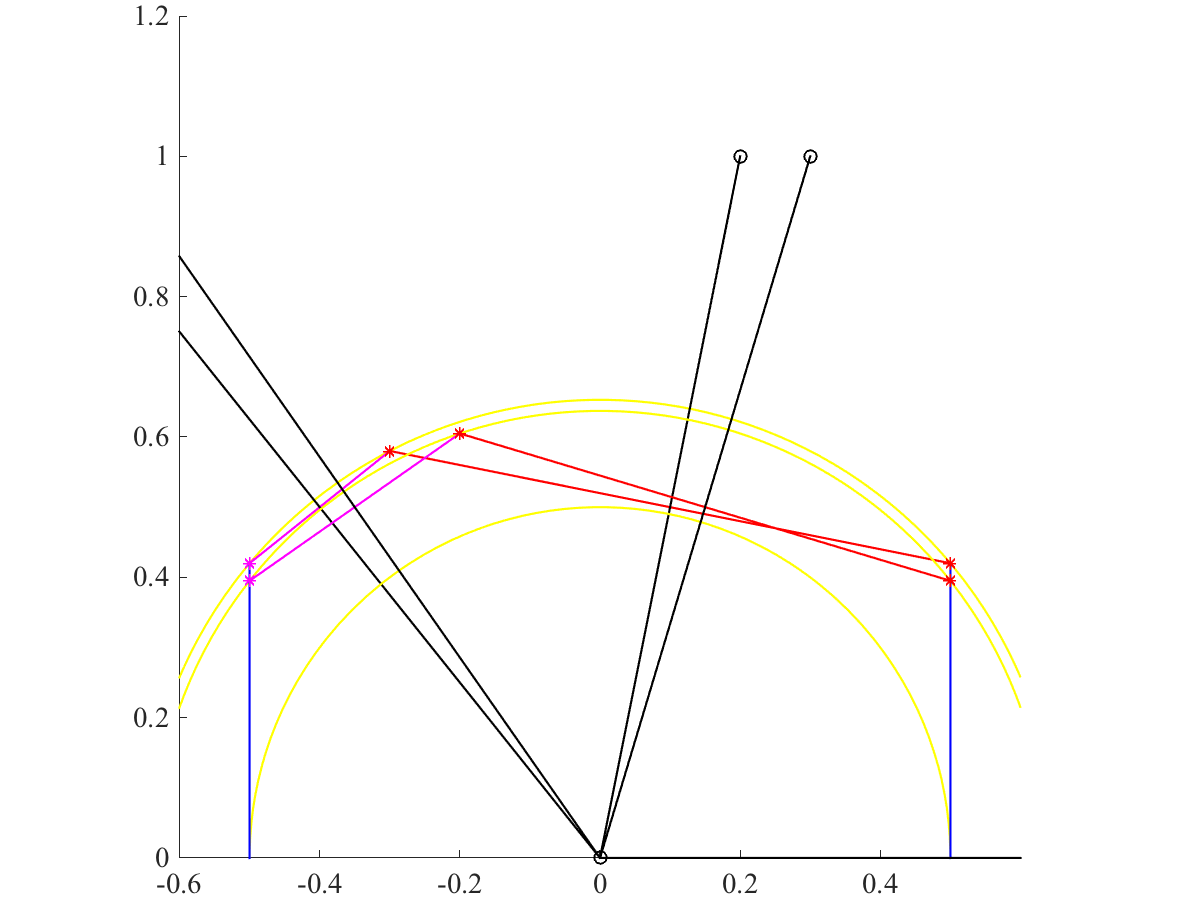}
\caption{An illustration of how the Dirichlet-Voronoi cell, $D_{a,b}$, changes between $(a,b) = (0.2,1)$ and $(a,b)=(0.3,1)$. }
\label{fig:ChangeDcell}
\end{center}
\end{figure}

\bigskip

\noindent \underline{Integral 1.}
Let $x = - \frac{1-a}{2\sqrt b} - \frac{a}{2\sqrt{b}} z$. Making this change of variables in $I_1$, we obtain 
$$
I_1 =  \frac{1}{b} \int_{-1}^{1} \frac{a (1 - a - az)}{4 b} f\left(\frac{b^2 + (1-a)^2}{4 b^3} ( b^2 + a^2 z^2)\right) \  dz
$$
Define the argument of $f$,
$$
A_1(z) := \frac{b^2 + (1-a)^2}{4 b^3} ( b^2 + a^2 z^2). 
$$
Note that $A_1(z)$ is even in $z$, so the term that is linear in $z$ is zero. Thus, we can write
$$
I_1 = \frac{1}{16b^2} \int_{-1}^{1} 4 a (1 - a) \  f\left( A_1(z) \right) \ (1-z) \ dz. 
$$
It will later be clear why we want the coefficient to contain the term $(1-z)$. 

\bigskip

\noindent \underline{Integral 2.} Let $x = \frac{4 a -1}{4 \sqrt{b}} + \frac{1}{4 \sqrt{b}}z$. Making this change of variables in $I_2$, we obtain 
$$
I_2 =  \frac{1}{16 b^2} \int_{-1}^{1}  f\left( A_2(z)  \right) \ (1-z) \  dz, 
$$
where the argument of $f$ is defined 
$$
A_2(z) :=\frac{a^2 + b^2}{16 b^3} \left(  4 b^2 + (1 - 2 a - z )^2 \right) . 
$$

\bigskip

\noindent \underline{Integral 3.} Let $x = \frac{1+2 a}{4 \sqrt{b}} - \frac{1-2 a}{4 \sqrt{b}} z $. Making this change of variables in $I_3$, we obtain 
$$
I_3 =  \frac{1}{16 b^2} \int_{-1}^{1} (1-2a)^2  \ f\left( A_3(z) \right) \ (1-z) \  dz, 
$$
where the argument of $f$  is defined 
$$
A_3(z) :=\frac{a^2 + b^2}{16 b^3} \left(  4 b^2 + \left( (1 - 2a )  z - 1 \right)^2 \right) . 
$$

\bigskip

\noindent \underline{Putting the pieces together.}
Note that the coefficients in $I_1$ and $I_3$ sum to 1, 
$$
4 a(1-a) + (1-2a)^2 = 1. 
$$
Breaking $I_2$ into two pieces, we can write 
\begin{align}
 \partial_a \frac{1}{2}  J(a,b) &=  - I_1  + I_2 - I_3 \nonumber \\
 &= \left[ 4 a(1-a)  I_2 - I_1 \right] +  \left[ (1-2a)^2 I_2 - I_3 \right] \nonumber \\
 &= \frac{1}{16 b^2} \int_{-1}^{1} 4 a(1-a) \  \left[  f \left( A_2(z) \right) - f\left( A_1(z)\right)  \right]  \ (1-z) \  dz \nonumber \\
 & \qquad  + \frac{1}{16 b^2} \int_{-1}^{1} (1-2a)^2  \ \left[  f \left( A_2(z) \right) - f\left( A_3(z) \right)  \right]  \ (1-z) \  dz 
\label{eq:JaPieces}
 \end{align}
We will show that $\partial_a J(a,b) > 0$ by showing that both of the integrands in \eqref{eq:JaPieces} are positive. 
\bigskip

\noindent \underline{Claim:} For every $z\in (-1,1)$, $A_1(z) \geq A_2(z)$. 

We compute 
\begin{align*}
\partial_z^2 \left( A_1(z) - A_2(z) \right) &= \frac{b^2 + (1-a)^2}{2 b^3} a^2 - \frac{a^2 + b^2}{8 b^3} \\
&= \frac{1}{8 b^3} \left( 4 a^2 (1-2 a) - \left(a^2+b^2\right)\left( 1 - (2 a)^2 \right) \right)\\
&= \frac{1-2a }{8 b^3} \left( 4 a^2  - \left(a^2+b^2\right)\left( 1 + 2 a \right) \right)\\
&= - \frac{1-2a }{8 b^3} \left( 2a \left(a^2+b^2\right) + b^2 - 3 a^2  \right)\\
& \leq 0
\end{align*}
as $b^2 \geq 3 a^2$.
Since 
\begin{align*}
A_1(-1) - A_2 (-1) &=  \frac{b^2 + (1-a)^2}{4 b^3} \left(a^2 + b^2\right) -   \frac{a^2 + b^2 }{4 b^3} \left(b^2 + \left( 1-a\right)^2 \right)  
= 0, \\
A_1(1) - A_2 (1) &=\frac{(1-2a) \left(a^2 + b^2 \right) }{4 b^3} \geq 0,
\end{align*}
 and $A_1(z) - A_2(z)$ is a concave function, it follows that $A_1(z) - A_2(z) \geq 0$ for all $z \in (-1,1)$. 

\bigskip

\noindent \underline{Claim:} For every $z\in (-1,1)$, $A_3(z) > A_2(z)$. 

We compute
\begin{align*}
A_3(z) &= \frac{a^2 + b^2}{16 b^3} \left[  4 b^2 + \left( (1 - 2a )  z - 1 \right)^2 \right] \\
&= \frac{a^2 + b^2}{16 b^3} \left[  4 b^2 + \left( 1 - 2a - z  \right)^2 + 4 a (1-a) (1-z^2) \right] \\
&> \frac{a^2 + b^2}{16 b^3} \left[  4 b^2 + \left( 1 - 2a - z  \right)^2  \right] \\
& = A_2(z). 
\end{align*}

\bigskip

Since $f$ is a decreasing function, these two claims show that the integrands in \eqref{eq:JaPieces} are strictly positive for every $z\in(-1,1)$. Thus,  $ \partial_a  J(a,b)>0$. 
\end{proof}

\begin{lem}   \label{lem:Jb<0}
Let $a = \frac{1}{2}$ and $b > \frac{\sqrt 3}{2}$ and $f$ a positive and decreasing function. Then $\partial_b J(a,b) < 0$. 
\end{lem}

\begin{proof}
In \eqref{eq:Jb}, we computed the partial derivative of $J(a,b)$ with respect to the parameter $b$. For $a = \frac{1}{2}$, we have $x_2 = 0$, 
$$
y_2(x) = \frac{1}{2\sqrt{b} } \left( \frac{1}{4b} + b \right) - \frac{1}{2b} x = y_1(-x), 
$$
and  
$$
\partial_b y_2(x) = \frac{1}{4\sqrt{b}} \left( 1 - \frac{3}{4b^2} \right) + \frac{1}{2b^2}x = (\partial_b y_1)(-x). 
$$
We obtain
\begin{align}
 \partial_b \frac{1}{2}  J(a,b)  
=&  -\frac{1}{2 b^{\frac 3 2}} \int_{0}^{y_2(x_1)} f\left(\frac{1}{4b} + y^2\right)  dy  \nonumber \\
& +\int_{-x_1}^{0} f\left( x^2 + y_2^2(-x) \right) \left( \partial_b y_2(-x) \right) dx 
 + \int_{0}^{x_1} f\left(x^2 + y_2^2(x)\right) \left( \partial_b y_2(x) \right) dx \nonumber \\
=&  -\frac{1}{2 b^{\frac 3 2}} \int_{0}^{y_2(x_1)} f\left(\frac{1}{4b} + y^2\right)  dy   + 2 \int_{0}^{x_1} f\left(x^2 + y_2^2(x)\right) \left( \partial_b y_2(x) \right) dx \label{eqref:K2}
\end{align}
We denote the second integral in \eqref{eqref:K2}, by 
$$
K_2 : = 2 \int_{0}^{x_1} f\left(x^2 + y_2^2(x)\right) \left( \partial_b y_2(x) \right) dx. 
$$
Making the change of variables, 
$$
z = \frac{4 b^2 -1}{2b} \left( x - \frac{x_1}{2}\right), 
$$
we compute 
$$
x^2 + y_2^2(x) =   \frac{1+4b^2}{16 b} + \frac{1+4b^2}{ (4b^2 -1 )^2} z^2 
\qquad \quad 
\text{and}
\qquad \quad 
\partial_b y_2(x) = \frac{4 b^2 -1}{16 b^{\frac 5 2}} + \frac{z}{b(4 b^2-1)}, 
$$
to obtain
\begin{align*}
K_2 &= 2 \int_{-z_1}^{z_1} f\left( \frac{1+4b^2}{16 b} + \frac{1+4b^2}{ (4b^2 -1 )^2} z^2 \right) \left( \frac{4 b^2 -1}{16 b^{\frac 5 2}} + \frac{z}{b(4 b^2 - 1)} \right) \frac{2b}{4b^2 -1}dz
\end{align*}
where $z_1 := \frac{4 b^2 -1}{ 8 b^{\frac 3 2}}$. The term involving the $z$ in the square brackets is an odd function and the term involving the constant is even so we have 
$$
K_2 = \frac{1}{2 b^{\frac 3 2}}   \int_{0}^{z_1} f\left( \frac{1+4b^2}{16 b} + \frac{1+4b^2}{ (4b^2 -1 )^2} z^2 \right) dz 
$$
Finally noting that  $z_1 = y_2(x_1)$, we obtain from \eqref{eqref:K2},
\begin{equation*}
 \partial_b \frac{1}{2}  J(a,b)  = \frac{1}{2 b^{\frac 3 2}}   \int_{0}^{z_1} f\left( \frac{1+4b^2}{16 b} + \frac{1+4b^2}{ (4b^2-1 )^2} z^2 \right) - f\left(\frac{1}{4b} + z^2\right)   dz .
\end{equation*}
But for $b\geq \frac{\sqrt{3}}{2}$, a simple calculation shows that 
$$
\frac{1}{4b} + z^2 \leq  \frac{1+4b^2}{16 b} + \frac{1+4b^2}{ (4b^2-1 )^2} z^2 ,  \qquad  z \in (0,z_1). 
$$
Thus, if $f$ is a positive and decreasing function, this implies that $ \partial_b   J(a,b) < 0$ for $a= \frac{1}{2}$ and $b > \frac{\sqrt{3}}{2}$. \end{proof}

\bigskip

\begin{proof}[Proof of Theorem \ref{thm:main}]  The preceding results are now put together as follows. We start with a flat torus, $T_{a,b}$ with $(a,b)\in U$. By Lemma \ref{lem:Ja>0}, by  continually rearranging the torus by increasing $a$, the objective function is strictly increasing. We stop when $a= \frac{1}{2}$. We then continually rearrange the torus by decreasing $b$ and by Lemma \ref{lem:Jb<0}, the objective function is again strictly increasing. We stop when $b = \frac{\sqrt{3}}{2}$. Thus the maximum is attained by the equilateral torus. This completes the proof. 
\end{proof}

\bigskip

\begin{rem}
Interestingly, the square torus is a saddle point of $J(a,b)$. This is a similar phenomenon to that observed by spectral properties of the Laplacian on a Bravais Lattice, see \cite{OM1}.  
\end{rem}

\appendix
\section{Proof of Proposition \ref{prop:LatticeParam}} \label{sec:LatPara}

 Consider an arbitrary lattice with unit volume. 
We first choose the basis vectors so that the angle between them is acute.  After a suitable rotation and reflection, we can let the shorter basis vector (with length $\frac{1}{\sqrt b}$) be parallel with the $x$ axis and the longer basis  vector (with length $ \sqrt{ \frac{a^2}{b} + b} = \sqrt{\frac{1}{b} \left( a^2 + b^2 \right)} \geq \sqrt{ \frac{1}{b}}$) lie in the first quadrant. 
Multiplying on the right by a unimodular matrix, $ \begin{pmatrix} 1 & 1 \\ 0 & 1 \end{pmatrix}$, we compute 
$$
\begin{pmatrix} \frac{1}{\sqrt b} & \frac{a}{\sqrt b} \\ 0 & \sqrt b \end{pmatrix} 
 \begin{pmatrix} 1 & 1 \\ 0 & 1 \end{pmatrix} = 
\begin{pmatrix} \frac{1}{\sqrt b} & \frac{a+1}{\sqrt b} \\ 0 & \sqrt b \end{pmatrix} . 
$$
Since this is equivalent to taking $a\mapsto a+1$, it follows that we can identify the lattices associated to the points $(a,b)$ and $(a+1,b)$. Thus, we can  restrict the parameter $a$ to the interval 
$\left[ 0, 1/2  \right]$ by symmetry.  For a complete picture of this restriction and how the symmetry naturally arises, see \cite[Proposition $3.2$ and Figure $3$]{LapBel2013}.  For more on flat tori, see also \cite{milnor1964eigenvalues,giraud2010hearing,laugesen2011sums}. $\square$

\section{A Translation of the Proof of The Fejes T\'oth Moment Lemma}
\label{app:momlem}

Here we translate and discuss  L\'aszl\'o Fejes T\'oth's  proof of his  Moment Theorem, which was published in German \cite{toth1953} and Russian \cite{toth1953russianbook}. Other references for the Moment Theorem include \cite{toth1973} where a sum of moments theorem is proved, \cite{gruber1999} where the sum of moments theorem is proved using the Moment Lemma below, \cite{BoroczkyCsikos2010} where a new version of the Moment Theorem appears with quadratic forms instead of moments, and \cite{Gruber2007Book} which contains a very nice summary of Moment type results.  See also \cite{bourne2014optimality} for an application of the Moment Theorem to block copolymer structures.  This is by no means a complete list of relevant references for the proofs and development of the Moment Theorem, but gives some idea of its long and important history.

We present also L\'aszl\'o Fejes T\'oth's  proof of his  Moment Lemma, as given in \cite{Toth1963}. This Lemma is also stated, but not proven in \cite{toth1973,gruber1999,Gruber2007Book}.  Some of the arguments and notation have been slightly modified from the original proof for clarification.

\begin{thm}{The Moment Theorem, \cite[p.80]{toth1953}}  \label{prop:MomThm}
Assume $f\colon [0,\infty) \to \mathbb R$ is a non-increasing function.  
Let $\mathbf{p}_1,\ldots \mathbf{p}_n$ be $n$ points in the plane and $C\subset \mathbb R^2$ a convex hexagon. Let  $d(\mathbf{p}) = \min_i \|  \mathbf{p} -  \mathbf{p}_i \|$  be the distance from the point $\mathbf{p}$ to the set $\{\mathbf{p}_i\}_{i=1}^n$. 
It holds that 
\begin{equation} \label{eq:RegHex}
\int_C f\left(d (\mathbf{p}) \right) \ d\mathbf{p} 
\ \leq \ 
n \cdot \int_\sigma  f\left( \| \mathbf{p} \| \right) \ d\mathbf{p} 
\end{equation}
where $\sigma $ is a regular hexagon with volume $  \frac { |C| } {n} $ and center $\mathbf o$.
\end{thm}

T\'oth's proof of Proposition \ref{prop:MomThm} relies on the Voronoi partition of the convex set $C\subset \mathbb R^2$ by the points $\{ \mathbf{p}_i \}_{i=1}^n$ and a careful analysis of the resulting partition components. Namely a rearrangement argument is used to compare each partition component to a ball which circumscribes the regular hexagon $\sigma$. The proof uses the following three lemmas, which we first prove. We note that Lemmas \ref{lem:Lemma1} and \ref{lem:Lemma2} are also proven in \cite{Imre1964} and stated (but not proven) in \cite{Toth1963}.

\begin{lem}{\cite[Lemma 1, p.82]{toth1953}} \label{lem:Lemma1}
Assume $f\colon [0,\infty) \to \mathbb R$ is a non-increasing function. 
Let $K\subset \mathbb R^2$ be a ball  centered at $\mathbf o$. Let $S= S(s)$ be a circular segment parameterized by its volume, $s = |S|$. Then the function 
$$
\omega(s) := \int_{S(s)} f\left( \| \mathbf{x} \| \right) \ d\mathbf{x}, \qquad 0\leq s < |K|/2
$$
is convex.
\end{lem}
\begin{proof}
Consider the circular segments with volumes $s_1 < s_1^*< s_2 < s_2^*$ with $s_i^* - s_i = \epsilon$ for $i=1,2$. Let $\Delta S_i$ for $i=1,2$ denote the (nearly trapezoidal) regions $S(s_i^*) \setminus S(s_i)$. 
From Figure \ref{fig:Fig78} it is clear that corresponding points $\mathbf x_1\in \Delta S_1$ and $\mathbf x_2\in \Delta S_2$ satisfy  $\|  \mathbf{x}_1\| > \| \mathbf{x}_2 \|$. Noting that $|\Delta S_2| = |\Delta S_1|$ and using the monotonicity of $f$, we  have that 
$$
\omega(s_1^*) - \omega(s_1) 
\ = \ \int_{\Delta S_1} f\left( \| \mathbf{x} \| \right) \ d\mathbf{x}
\ \leq \  \int_{\Delta S_2} f\left( \|  \mathbf{x} \| \right) \ d\mathbf{x}
\ = \  \omega(s_2^*) - \omega(s_2) .
$$
Dividing both sides by $\epsilon$ and taking the limit $s_i^* \to s_i$ for $i=1,2$, {\it i.e.}, $\epsilon \to 0$, 
we have that $\omega'(s_1) \leq \omega'(s_2)$ which shows that $\omega$ is a convex function. 
\end{proof}

\begin{figure}[h!]
\begin{center}
\begin{tikzpicture}[thick,scale=3.0]
    \draw [color=black] circle(1cm);
    \draw (0,0) node[circle,fill,inner sep=1pt,label=left:$\mathbf{o}$] (origin){};
      
   \draw (origin) (60:1cm) coordinate (s1){};
   \draw (origin) (-60:1cm) coordinate  (s2){};   
   \draw (origin) (0:.5cm) coordinate (s3){};
   \draw (origin) (0:1cm) coordinate (s4){};
   
   \draw [color=blue,name path=line3] (s1) -- (s2);
   \draw [fill=blue,opacity=.1] (s1) -- (s2) arc [radius=1, start angle=-60, end angle=60];  
   \draw [color=blue] (origin) (-8:.8cm) coordinate [label=$S(s)$] (S); 
   \draw (origin) (140:.8cm) coordinate [label=$K$] (K); 

\end{tikzpicture}
\begin{tikzpicture}[thick,scale=3.0]
    \draw [color=black] circle(1cm);
    \draw (0,0) node[circle,fill,inner sep=1pt,label=above left:$\mathbf{o}$] (origin){};
    
    \draw (origin) (140:.8cm) coordinate [label=$K$] (K); 
    
    \draw (origin) (80:1cm) node[circle,fill,inner sep=1pt] (s1a){};
    \draw (origin) (-80:1cm) node[circle,fill,inner sep=1pt] (s1b){};
    \draw (origin) (77:1cm) node[circle,fill,inner sep=1pt] (s1pa){};
    \draw (origin) (-77:1cm) node[circle,fill,inner sep=1pt] (s1pb){};
    \draw (s1a) -- (s1pa) -- (s1pb) -- (s1b) -- (s1a);
    \draw (origin) (90:.7cm) coordinate [label=$\Delta S_2$] (ds2); 
    \draw (origin) (-59:.4cm) node[circle,fill,inner sep=1pt,label=left:$\mathbf{x}_2$] (x2){};     
    \draw (origin) -- (x2);
      
    \draw (origin) (50:1cm) node[circle,fill,inner sep=1pt] (s2a){};
    \draw (origin) (-50:1cm) node[circle,fill,inner sep=1pt] (s2b){};
    \draw (origin) (46:1cm) node[circle,fill,inner sep=1pt] (s2pa){};
    \draw (origin) (-46:1cm) node[circle,fill,inner sep=1pt] (s2pb){};
    \draw (s2a) -- (s2pa) -- (s2pb) -- (s2b) -- (s2a);
    \draw (origin) (10:.85cm) coordinate [label=$\Delta S_1$] (ds1); 
    \draw (origin) (-21:.72cm) node[circle,fill,inner sep=1pt,label=right:$\mathbf{x}_1$] (x1){};     
    \draw (origin) -- (x1);

\end{tikzpicture}

\caption{Illustrations for the statement and proof of Lemma \ref{lem:Lemma1}. See \cite[Fig.78]{toth1953}. }
\label{fig:Fig78}
\end{center}
\end{figure}

\begin{figure}[h!]
\begin{center}
\begin{tikzpicture}[thick,scale=3.0]
    \draw [color=black] circle(1cm);
    \draw (0,0) node[circle,fill,inner sep=1pt,label=left:$\mathbf{o}$] (origin){};
    
    \draw (origin) (45:1cm) node[circle,fill,inner sep=1pt,label=above right:$\mathbf{a}'$] (Ap){};
    \draw (origin) (45:.8cm) node[circle,fill,inner sep=1pt,label= right:$\mathbf{a}$] (A){};
    \draw (origin) -- (A);
    
    \draw (origin) (-45:1cm) node[circle,fill,inner sep=1pt,label=below right:$\mathbf{b}'$] (Bp){};
    \draw (origin) (-45:.4cm) node[circle,fill,inner sep=1pt,label=below left:$\mathbf{b}$] (B){};
    \draw (origin) -- (B);
    
   \draw [color=red] (Ap) -- (A);
   \draw [name path=line1,color=red] (A) -- (B);
   \draw [name path=line2,color=red] (B) -- (Bp);    
   \draw [color=red] (Bp) arc [radius=1, start angle=-45, end angle=45] -- (Ap);
   \draw [color=red] (origin) (-0:.8cm) coordinate [label=$R$] (R); 
   
   \draw (origin) (60:1cm) node[circle,fill,inner sep=1pt,] (s1){};
   \draw (origin) (-60:1cm) node[circle,fill,inner sep=1pt,] (s2){};   
   \draw [color=blue,name path=line3] (s1) -- (s2);
   \draw [fill=blue,opacity=.1] (s1) -- (s2) arc [radius=1, start angle=-60, end angle=60];  
   \draw [color=blue] (origin) (-54:1cm) coordinate [label=$S$] (S); 
   
   \path [name intersections={of=line1 and line3,by=C}];
   \node [fill=black,inner sep=1pt,label=right:$\mathbf{c}$] at (C) {};

   \path [name intersections={of=line2 and line3,by=D}];
   \node [fill=black,inner sep=1pt,label=right:$\mathbf{d}$] at (D) {};
\end{tikzpicture}

\caption{Illustrations for the statement and proof of Lemma \ref{lem:Lemma2}. See \cite[Fig.79]{toth1953}. }
\label{fig:Fig79}
\end{center}
\end{figure}

\begin{lem} \label{lem:Lemma2}
Assume $f\colon [0,\infty) \to \mathbb R$ is a non-increasing function. 
Let $\mathbf a$ and $\mathbf b$ be two points contained in the ball $K$ 
centered at 
$\mathbf o$. 
Let $\mathbf{a}'$ and $\mathbf{b}'$ be the points at the respective intersections between the 
rays $\overrightarrow{ \mathbf{o} \mathbf{a}}$ and $\overrightarrow{ \mathbf{o} \mathbf{b}}$ with $\partial K$. 
Let $R\subset K$ denote the region enclosed by the path defined by the segments $\overline{\mathbf{a}' \mathbf{a}}$, $\overline{\mathbf{a} \mathbf{b} }$, $\overline{ \mathbf{b} \mathbf{b}' }$, and the circular arc $\angle \mathbf{b}' \mathbf{a}' $. 
Let $S = S(s)$ denote the circular segment with volume $ s = |R| $. 
It holds that 
$$
\omega (s) := 
\int_S f\left( \|  \mathbf{x} \| \right)  \ d\mathbf{x} 
\  \leq \ 
\int_R f\left( \|   \mathbf{x} \| \right)  \ d\mathbf{x} .
$$
\end{lem}
\begin{proof}
As in Figure \ref{fig:Fig79}, position the circular segment, $S$, so that there are two points, $\mathbf{c}$ and $\mathbf{d}$, at $\partial S \cap \partial R$ equidistant to $\mathbf{o}$. 
Since, aside from  $\mathbf{c}$ and $\mathbf{d}$, each point of the region $R \setminus S$ is closer to $\mathbf{o}$ then each point in $S \setminus R$, using the monotonicity of $f$, we have the inequality 
$$
\int_R 
\ =  \ 
\int_{R \cap S} + \int_{R \setminus S} 
\ \geq \ 
\int_{S \cap R} + \int_{S \setminus R} 
\ = \ 
\int_{S},
$$
where the integrand  for each of these integrals is $f\left( \| \mathbf{x} \| \right)  \ d\mathbf{x}$. 
\end{proof}

The following Lemma is a summary of a discussion in  \cite[p.14--16]{toth1953}. Since this discussion contains quite a bit of supplementary information, we more closely follow  \cite{Bambah1952}. 
\begin{lem} \label{lem:Lemma3} Consider the Voronoi partition of a convex hexagon $C\subset \mathbb R^n$ by the points $\{ \mathbf{p}_i \}_{i=1}^n \subset C$. If $P_i$ denotes the number of vertices in the Voronoi cell for point $\mathbf{p}_i$ and $N = \sum_{i=1}^n P_i$, then $ N \leq 6 n$.  
\end{lem}
\begin{proof} 
Let $C = \cup_{i=1}^n T_i$ denote the Voronoi partition. Denote the unbounded exterior polygonal region by $T_0$. We view the  partition interfaces together with $\partial C$ as a connected planar graph with $V$ vertices, $E$ edges (sides), and $n+1$ regions (including the exterior region). Euler's characteristic then states that $$V - E + n =1. $$ 
Let $d_\ell$, for $\ell = 1,\ldots,V$   be the degree of each vertex. Note that, except at the vertices of $C$, $d_\ell$ is at least 3, giving 
$$
\sum_{\ell=1}^V d_\ell \geq 3(V-6) + 2\cdot 6 = 3V - 6.
$$
Since each edge connects exactly 2 vertices, $\sum_\ell d_\ell = 2 E$ (a.k.a. the Handshake lemma) and we have  
$$
2E \geq 3V-6 = 3( E - n +1) - 6 
 \qquad \implies \qquad E \leq 3n + 3.
 $$
Polygon $T_i$ has $P_i$ sides for $i=0,\ldots n$.  Since each side is incident to exactly two regions, $2 E = \sum_{i=0}^n P_i$. Since the exterior polygon has at least 6 sides, 
$$
N = \sum_{i=1}^n P_i  \leq \sum_{i=0}^n P_i - 6 = 2E - 6 \leq 2(3n+3) -6 = 6n 
$$
as desired.
\end{proof}

\begin{figure}[h!]
\begin{center}
\resizebox{6.5cm}{!}{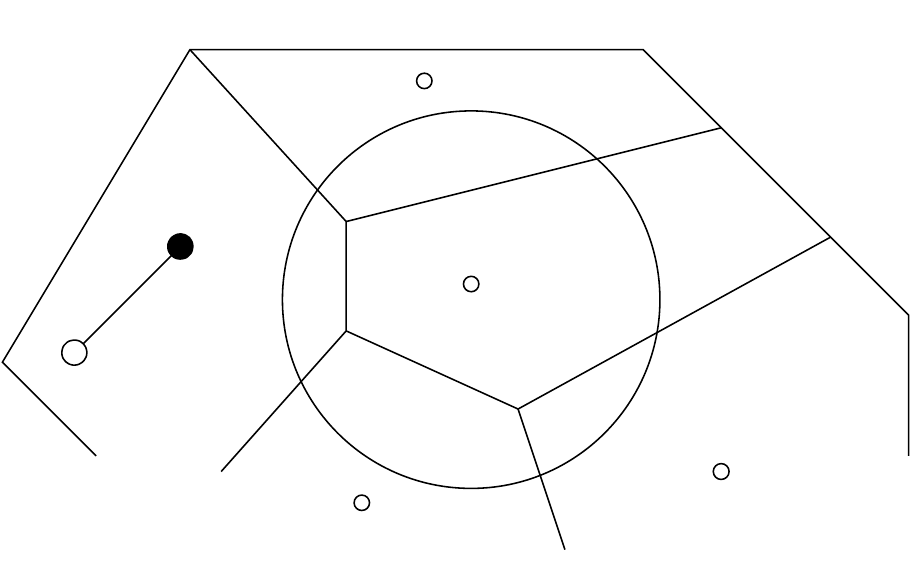} 
\begin{tikzpicture}[thick,scale=2.0]
    \draw [color=black] circle(1cm);
    \draw (0,0) node[circle,fill,inner sep=1pt] (origin){};
    \draw (origin) (160:1.1cm) coordinate [label=$K$] (K);
    \draw (origin) (45:.4cm) coordinate [label=$\sigma$] (); 
    \draw (0:1cm) \foreach \x in {60,120,...,359} {-- (\x:1cm) } -- cycle (90:1cm);      
\end{tikzpicture}
\qquad
\begin{tikzpicture}[thick,scale=2.0]
    \draw [color=black] circle(1cm);
    \draw (0,0) node[circle,fill,inner sep=1pt] (origin){};    
    \draw (origin) (140:.8cm) coordinate [label=$K$] (K); 
           
   \draw (origin) (60:1cm) coordinate (s1){};
   \draw (origin) (-60:1cm) coordinate (s2){};
   \draw [color=blue] (s1) -- (s2);
   \draw [fill=blue,opacity=.1] (s1) -- (s2) arc [radius=1, start angle=-60, end angle=60];  
   \draw [color=blue]  (.8cm,0) coordinate [label=$H$] (S); 

   \draw (origin) (70:1cm) coordinate (s3){};
   \draw (origin) (-70:1cm) coordinate (s4){};   
   \draw [color=red] (s3) -- (s4);
   \draw [fill=red,opacity=.1] (s1) -- (s2) -- (s4) -- (s3);  
   \draw [color=red] (.42cm,0) coordinate [label=$U$] (S);   
\end{tikzpicture}

\caption{Illustrations for the proof of Theorem \ref{prop:MomThm}. See \cite[Fig.80]{toth1953}. }
\label{fig:Fig80}
\end{center}
\end{figure}

\begin{proof}[Proof of Theorem \ref{prop:MomThm}] 
Without loss of generality, we can take that the points $\mathbf{p}_1,\ldots \mathbf{p}_n$ to lie in $C$,  since the left hand side of  \eqref{eq:RegHex} increases under projection onto $C$. We consider the Voronoi partition $\cup_{i=1}^n T_i$ of $C$ by the centers $\{\mathbf{p}_i\}_{i=1}^n$, {\it i.e.}, the partition  of $C$ into convex polygons $T_1, \ldots T_n$ such that for all $\mathbf{x} \in T_i$, 
$$
\| \mathbf{x} - \mathbf{p}_{i} \| \leq \| \mathbf{x} -  \mathbf{p}_{j}\|,  \quad \forall j\neq i.
$$  
See Figure~\ref{fig:Fig80}(left).
We decompose the integral on the left hand side of \eqref{eq:RegHex}:
$$
\int_C f\left(d (\mathbf{p}) \right) \ d\mathbf{p}   \ = \
\sum_{i=1}^n \int_{T_i} f\left( \|\mathbf{p}- \mathbf{p}_i\| \right) \ d\mathbf{p} . 
$$

We label the $P_i$ vertices of  polygon $T_i$ in cyclic order by $\mathbf{e}_1, \ldots, \mathbf{e}_{P_i}$.
To each vertex $\mathbf{p}_i$, we take a circle $K_i $ with center $\mathbf{p}_i$ of size such that (a translation to the origin of) $K_i$ circumscribes the regular hexagon, $\sigma$, with  volume $ |\sigma| = |C|  /n$. 
Let $R^i_\ell \subset  K_i \setminus T_i$ for $\ell = 1,\ldots,P_i$ be a region associated to the $\ell$-th edge of  Voronoi  cell $i$ as illustrated in Figure \ref{fig:Fig80}(left). Of course, we allow for $R^i_\ell$ to be empty. Additionally we define $T_i' = T_i \setminus K_i$, which could also be empty. 
We then decompose the integral 
$$
\int_{T_i} = \int_{K_i} + \int_{T_i'} - \sum_{\ell=1}^{P_i} \int_{R^i_\ell}
$$
where the integrands are all $f\left( \| \mathbf{p}- \mathbf{p}_i \| \right) \ d\mathbf{p}$. Hereinafter we will frequently suppress  integrands to simplify exposition. Note that the first integral is the same for all $i=1,\ldots,n$, 
$$
\int_{K_i} f\left( \| \mathbf{p}- \mathbf{p}_i\| \right) \ d\mathbf{p} 
\ = \
\int_{K} f\left( \| \mathbf{p}  \| \right) \ d\mathbf{p} . 
$$

Next, we consider the rearrangement of each region $R^i_\ell$ into a circular segment $S = S(s^i_\ell)$ with volume $s^i_\ell = |S| = |R^i_\ell |$. Using  Lemma \ref{lem:Lemma2},  we have
$$
\int_{R^i_\ell} f\left( \| \mathbf{p} - \mathbf{p}_i \| \right) \ d\mathbf{p} 
\ \geq \ 
\omega \left( |R^i_\ell | \right),
$$ 
giving 
$$
\int_{T_i}  
\ \leq \ 
 \int_{K} + \int_{T_i'} - \sum_{\ell=1}^{P_i} \omega \left( |R^i_\ell | \right). 
$$
Now summing over $i=1,\ldots, n$, we have that 
$$
\int_C f\left( d (\mathbf{p}) \right) \ d\mathbf{p} 
\ = \ \sum_{i=1}^n \int_{T_i}   
\ \leq  \
n \int_{K} 
\ + \ \sum_{i=1}^n \int_{T_i'} 
\ - \ \sum_{i=1}^n \sum_{\ell=1}^{P_i} \omega \left( |R^i_\ell |  \right). 
$$
Let $N = \sum_{i=1}^n P_i$. 
Using the convexity of $\omega$ (Lemma \ref{lem:Lemma1}) and Jensen's inequality, we have 
$$
\int_C f\left(d (\mathbf{p}) \right) \ d\mathbf{p}   
\ \leq  \
n \int_{K} 
\ + \ \sum_{i=1}^n \int_{T_i'} 
\ - \ 
N \cdot  \omega \left( \frac{1}{N} \sum_{i=1}^n \sum_{\ell=1}^{P_i}  |R^i_\ell |  \right). 
$$
Letting $T' = \cup_{i=1}^n T_i'$, we  have that 
$$ 
|C| \ = \ 
n\cdot |K| \ + \ |T'| \  - \ \sum_{i=1}^n \sum_{\ell=1}^{P_i}  |R^i_\ell | 
\quad \implies \quad 
\frac{1}{N} \sum_{i=1}^n \sum_{\ell=1}^{P_i}  |R^i_\ell | = \frac{n\cdot |K| \ + \ |T'| \ - |C|}{N} . 
$$ 
Since $\omega(0) = 0$ and $\omega$ is an increasing, convex function it follows that $\frac{\omega(x)}{x} $ is a continuous and increasing function. Using Lemma \ref{lem:Lemma3}, we have that $N \leq 6 n$ which implies  
\begin{equation} \label{eq:AlmostThere}
\int_C f\left(d (\mathbf{p}) \right) \ d\mathbf{p}   
\ \leq  \
n \int_{K} 
\ + \ \sum_{i=1}^n \int_{T_i'} 
\ - \ 
6n \cdot  \omega \left( \frac{  n\cdot |K| + |T'| - |C| }{6n}  \right). 
\end{equation}
Let $H$ be one of the circular segments in $K \setminus \sigma$. Recalling $  |C|  = |\sigma| n$,  we have 
$$
H = S( |H| ) =  S\left( \frac{   |K| - |\sigma| }{6} \ \right) = S\left(  \frac{  n\cdot |K| - |C| }{6n}  \right).
$$ 
We define $U$ to be the set difference of two circular segments of $K$, 
$$U = S\left( \frac{  n\cdot |K | + |T'| - |C| }{6n} \ \right) \setminus H, 
$$ 
with volume $|U| = |T'| / 6n $;  see Figure \ref{fig:Fig80}(right). We have
$$
\omega \left( \frac{  n\cdot |K| + |T'| - |C| }{6n} \ \right) 
\ = \ 
\omega \left( |H| \right) +  \int_{U} f\left( \| \mathbf{p} \| \right) \ d\mathbf{p}. 
$$
From \eqref{eq:AlmostThere}, we have that 
\begin{align}
\label{eq:AlmostThere2}
\int_C f\left(d (\mathbf{p}) \right) \ d\mathbf{p}   
 \ \leq  & \  
n \left[ \int_{K}  f\left( \| \mathbf{p}  \| \right) \ d\mathbf{p} - 6 \cdot \omega \left( |H| \right) \right] \\
\nonumber
&  + \ \sum_{i=1}^n \int_{T_i'} f\left( \|\mathbf{p}- \mathbf{p}_i \| \right) \ d\mathbf{p}
\ - \ 
6n \cdot  \int_{U} f\left( \| \mathbf{p} \| \right) \ d\mathbf{p}
\end{align}
Noting that $U \subset K$, we have $ f(r) = \min\{f\left( \|  \mathbf{p} \|  \right)\colon \mathbf{p} \in U\}$ where $r$ is the radius of $K$ and similarly $f(r) = \max\{f\left( \|  \mathbf{p} -  \mathbf{p}_i \| \right) \colon \mathbf{p} \in T_i'\}$. Thus, in the second line in the right hand side of  \eqref{eq:AlmostThere2}, we have 
$$
\sum_{i=1}^n \int_{T_i'} f\left( \| \mathbf{p}- \mathbf{p}_i \| \right) \ d\mathbf{p} - 
6n \cdot  \int_{U} f\left( \| \mathbf{p} -  \mathbf{p}_i \| \right) \ d\mathbf{p}
 \ \leq  \ 
 f(r) \cdot \left( |T'|  - 6 n |U| \right) = 0. 
$$
The first line in the right hand side of \eqref{eq:AlmostThere2} is exactly $n \cdot \int_\sigma  f\left( \| \mathbf{p}  \| \right) \ d\mathbf{p}$, which proves \eqref{eq:RegHex}.
\end{proof}

\begin{rem} \cite{gruber1999} states the Moment Theorem for convex polygons with $n=3,4,5,$ or 6 vertices. This generalized statement follows from Proposition \ref{prop:MomThm} by introducing false vertices for $n=3,4,$ or 5. 
\end{rem}

%

\begin{thm}[The Moment Lemma, \cite{Toth1963}] \label{thm:MomLemma}
Assume $f\colon [0,\infty) \to \mathbb R$ is a non-increasing function.  
Let  $C\subset \mathbb R^2$ be a convex $n$-gon and $\sigma$ the regular $n$-gon centered at the origin with $|C| = |\sigma|$. Then 
\begin{equation} \label{eq:MomLem}
\int_C f\left( \| \mathbf{p} \| \right) \ d\mathbf{p} 
\ \leq \ 
\int_\sigma  f\left( \| \mathbf{p} \| \right) \ d\mathbf{p} . 
\end{equation}
\end{thm}

\bigskip

The moment Lemma for hexagons follows from the Moment Theorem  when $n=1$. 
We also mention that Peter M. Gruber's proof of the Moment Theorem relies on the Moment Lemma  \cite{gruber1999,Gruber2007Book}. The following proof, which closely follows \cite{Toth1963}, is very similar to the proof of the Moment Theorem.

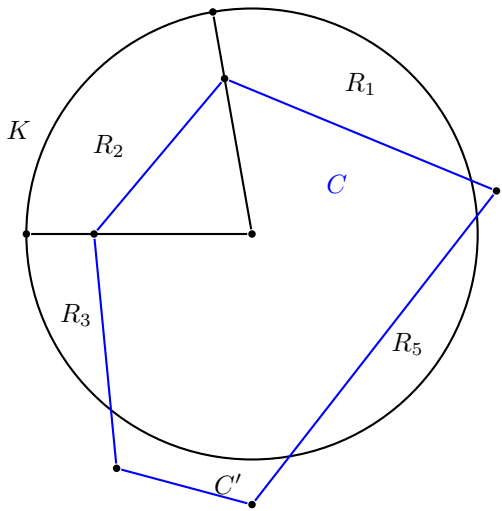
\begin{figure}[t!]
\begin{center}
\begin{tikzpicture}[thick,scale=3.0]
    \draw [color=black] circle(1cm);
    \draw (0,0) node[circle,fill,inner sep=1pt] (origin){};
    \draw (origin) (160:1.1cm) coordinate [label=$K$] (K);
    \draw [color=blue] (origin) (20:.4cm) coordinate [label=$C$] (); 
    
    \draw (origin) (100:.7cm) node[circle,fill,inner sep=1pt] (A){};
    \draw (origin) (100:1cm) node[circle,fill,inner sep=1pt] (A'){};
    \draw (origin) (10:1.1cm) node[circle,fill,inner sep=1pt] (B){};
    \draw (origin) (270:1.2cm) node[circle,fill,inner sep=1pt] (C){};
    \draw (origin) (240:1.2cm) node[circle,fill,inner sep=1pt] (D){};    
    \draw (origin) (180:1cm) node[circle,fill,inner sep=1pt] (D') {}; 
    \draw (origin) (180:.7cm) node[circle,fill,inner sep=1pt] (E){};
    \draw [color=blue,name path=line3] (A) -- (B) -- (C) -- (D)-- (E) -- (A) ;

    \draw (origin) -- (A'); 
    \draw (origin) (50:.75cm) coordinate [label=$R_1$] ();    
    \draw (origin) (155:.7cm) coordinate [label=$R_2$] ();
    \draw (origin) -- (D'); 
    \draw (origin) (210:.9cm) coordinate [label=$R_3$] ();    
    \draw (origin) (-40:.9cm) coordinate [label=$R_5$] ();    
    \draw (origin) (265:1.2cm) coordinate [label=$C'$] (K);

\end{tikzpicture}

\caption{Illustration for the proof of Theorem \ref{thm:MomLemma}.}
\label{fig:FigMomLemma}
\end{center}
\end{figure}

\begin{proof}[Proof of Theorem \ref{thm:MomLemma}.]
We may assume that $C$ contains the origin. Otherwise, by translating $C$ towards the origin the integral on the left side of \eqref{eq:MomLem} would increase. Let $K$ denote the circle that circumscribes $\sigma$. 
Define $R_\ell \subset K$, $\ell = 1,\ldots, n$ to be the regions as in Figure \ref{fig:FigMomLemma} so that $\cup_\ell R_\ell = K \setminus C$.  We allow some of the $R_\ell$ to be empty.   Define $C' = C \setminus K$. 
We can then write
$$
\int_C =  \int_K - \sum_\ell \int_{R_\ell} + \int_{C'}
$$
where the integrands here and below are understood to be $ f\left( \| \mathbf{p} \| \right) \ d\mathbf{p}$. 
Using Lemma \ref{lem:Lemma2}, we have that $\int_{R_\ell} \geq \omega ( | R_\ell| )$.  Using the convexity of $\omega$ (Lemma \ref{lem:Lemma1}) and Jensen's inequality, we then have that 
\begin{align*}
\int_C  \ \leq \
 \int_K - \sum_i \omega( |R_\ell| ) + \int_{C'} \  \leq  \ 
 \int_K - n \omega\left(  \frac{1}{n} \sum_i |R_\ell| \right) + \int_{C'} 
\end{align*}
Let $H$ be one of the  circular segments in $K \setminus \sigma$. 
Using the equality of the volumes
\begin{align*}
|C| = |K|  - \sum_\ell |R_\ell|  + |C'| 
\qquad \text{and} \qquad 
|\sigma| = |K| - n |H| 
\end{align*}
we obtain 
$$
\frac{1}{n} \sum_\ell |R_\ell| = \frac{1}{n} |C' | + |H|.
$$ 
As in Figure \ref{fig:Fig80}(right), let $U$ be the set difference between the two circular segments of $K$,  
$$
U = S\left( \frac{1}{n} |C' | + |H| \right) \setminus H, 
$$ 
we have that 
$$
\omega\left(  \frac{1}{n} \sum_\ell |R_\ell| \right) = \omega(|H|) + \int_U. 
$$
Using the fact that all points in $U$ are closer to the origin then points in $C'$ and the monotonicity of $f$, we have that 
$ n \int_U \geq \int_{C'}$. 
Finally, observing that $\int_\sigma = \int_K - n \cdot \omega(|H|)$, we conclude that 
\begin{align*}
\int_C  \ \leq \ 
\int_K - n \cdot \left(  \omega\left(  |H| \right) + \int_U \right) + \int_{C'} \ \leq \ 
\int_\sigma, 
\end{align*}
as desired. 
\end{proof}

\section*{Acknowledgments} The authors wish to thank Mikael Rechtsman and Sylvia Serfaty for very helpful conversations during the preparation of this manuscript.
BO was supported in part by U.S. NSF DMS-1461138.  JLM was supported in part by U.S. NSF DMS-1312874 and NSF CAREER Grant DMS-1352353.

{ \small
\printbibliography 
}

\end{document}